\documentclass[11pt, a4paper]{article}
\usepackage{amsmath,amssymb,amsthm,amsfonts,latexsym,graphicx,subfigure}
\usepackage{color}
\usepackage[all,import]{xy}
\usepackage[numbers,sort&compress]{natbib}
\usepackage{indentfirst}
\usepackage{fancyhdr}
\usepackage{bbm}
\usepackage{accents,cases}
\usepackage{algorithm}
\usepackage{algorithmic}
\usepackage{multirow}
\usepackage{enumerate}
\usepackage[colorlinks,
linkcolor=black,
anchorcolor=black,
citecolor=black
]{hyperref}
\usepackage{array}
\newcommand{\PreserveBackslash}[1]{\let\temp=\\#1\let\\=\temp}
\newcolumntype{C}[1]{>{\PreserveBackslash\centering}p{#1}}
\newcolumntype{R}[1]{>{\PreserveBackslash\raggedleft}p{#1}}
\newcolumntype{L}[1]{>{\PreserveBackslash\raggedright}p{#1}}

\DeclareMathOperator{\rank}{\ensuremath{rank}}
 \DeclareMathOperator{\sym}{\ensuremath{Sym}}
   \DeclareMathOperator{\kernel}{\ensuremath{ker}}
   \DeclareMathOperator{\cokernel}{\ensuremath{coker}}
   \DeclareMathOperator{\ext}{\ensuremath{Ext}}
   \DeclareMathOperator{\slope}{\ensuremath{slope}}
   \DeclareMathOperator{\maxslope}{\ensuremath{maxslope}}

\makeatletter
\def\wbar{\accentset{{\cc@style\underline{\mskip8mu}}}}

\makeatother

\numberwithin{equation}{section}

\theoremstyle{plain}
\newtheorem{theorem}{Theorem}[section]
\newtheorem{defn} [theorem] {Definition}
\newtheorem{lemma} [theorem] {Lemma}
\newtheorem{remark}[theorem]{Remark}

\newtheorem{cor}[theorem]{Corollary}
\newtheorem{pro}[theorem]{Proposition}

\begin{document}
\bibliographystyle{unsrt}

\title{Algebraic  surfaces with $p_g=q=1, K^2=4$ and  nonhyperelliptic Albanese fibrations of genus 4}
\author{Songbo Ling}
\date{}
\renewcommand{\thefootnote}{\fnsymbol{footnote}}
\maketitle
 \footnotetext{This work  has been supported by the NSFC (No. 11471020).}

\begin{abstract}	
In this paper we  study minimal algebraic surfaces with $p_g=q=1,K^2=4$ and nonhyperelliptic Albanese  fibrations of genus 4.   We construct for the first time a family of such surfaces as complete intersections of   type $(2,3)$ in a $\mathbb{P}^3$-bundle over an elliptic curve.  For the surfaces we construct here, the direct image of the canonical sheaf under the Albanese map    is decomposable (which   is a topological invariant property). 

 Moreover we prove that, all minimal surfaces with $p_g=q=1,K^2=4$ and nonhyperelliptic Albanese  fibrations of genus 4 such    that  the direct image of the canonical sheaf under the Albanese map    is decomposable are contained in our family.   As a consequence, we show    that these  surfaces   
constitute a 4-dimensional  irreducible subset $\mathcal{M}$ of  $\mathcal{M}_{1,1}^{4,4}$, the Gieseker moduli space of minimal surfaces with $p_g=q=1, K^2=g=4$.  Moreover,   the closure of $\mathcal{M}$ is an irreducible component of       $\mathcal{M}_{1,1}^{4,4}$.
\end{abstract}
	
	\section{Introduction}
 Minimal surfaces of general type with $p_g=q=1$ have attracted the interest of  many authors (e.g. \cite{Cat81,CC91,CC93,CP06,Hor81,Pig09,Pol09,Rit1,Rit2}). For these surfaces, by an inequality of Bombieri (cf. \cite{Bom73} Lemma 14) and the Bogomolov-Miyaoka-Yau inequality   one has $2\leq K^2\leq 9$.  By results of Mo\v{\i}\v{s}ezon \cite{Moi65},  Kodaira \cite{Kod68} and Bombieri \cite{Bom73},   these surfaces belong to a finite number  of families.

For such a surface $S$,  the Albanese map $f:S\rightarrow B:=Alb(S)$   is a fibration. Since the genus $g$ of a general    fibre of $f$ (cf. \cite{CP06} Remark 1.1) and $K_S^2$ are    differentiable invariants,  surfaces with different $g$  or  $K^2$  belong to  different connected components of the Gieseker moduli space. Hence  we can write  the Gieseker moduli space of minimal surfaces of general type with $p_g=q=1$ $\mathcal{M}_{1,1}$ as a disjoint union of its components $\sqcup _{K^2,g}\mathcal{M}_{1,1}^{K^2,g}$, where $\mathcal{M}_{1,1}^{K^2,g}$ corresponding to minimal surfaces with $p_g=q=1$ and fixed $(K^2,g)$. 
 We call $\mathcal{M}_{1,1}^{K^2,g}$ the Gieseker moduli space of  minimal surfaces of general type  with $p_g=q=1$ and fixed $(K^2,g)$.

	The case  $K^2=2$ has been accomplished    by Catanese \cite{Cat81} and Horikawa \cite{Hor81} independently: these surfaces have $g=2$ and  the moduli space is irreducible;  the case  $K^2=3$ has been studied completely by Catanese-Ciliberto \cite{CC91,CC93} and Catanese-Pignatelli \cite{CP06}: these surfaces have $g=2$ or $g=3$. Moreover,  there are  three irreducible connected components for surfaces with $g=2$ and one  irreducible connected component for  surfaces with $g=3$.

\vspace{2ex}
In the case $K^2=4$,
  Ishida \cite{Ish05} proved  that  $g\leq 4$ if  the general Albanese fibre  is hyperelliptic.  Moreover,   he also constructed such  surfaces with $g=3,4$  (see \cite{Ish05}). For $g=2$, there are many examples (e.g. see Catanese \cite{Cat98}  Example 7,8 and Pignatelli \cite{Pig09}). In particular,  Pignatelli \cite{Pig09} found eight disjoint irreducible  components of the moduli space under the assumption that the direct image of the bicanonical sheaf under the Albanese map is a direct sum of three line bundles.

    However, if the general Albanese fibre is nonhyperelliptic, we only know that $g\leq 6$ if  the direct image of the   canonical sheaf under the Albanese map  is   indecomposable (see Remark \ref{T2 is 0 for genus bigger than 4} (ii)).  For the case  $g=3$  there are many known examples (see e.g. \cite{Ish05,Pol09,Rit1,Rit2}\cite{Ling43}). In particular, under the assumption that  direct image of the canonical sheaf under the Albanese map    is decomposable, the author \cite{Ling43} found  two irreducible components of the moduli space,  one of  dimension 5 and the other of dimension 4.    But  when  $g\geq 4$, as far as we know,  no example is known.

\vspace{2ex}
In this paper, we   study minimal algebraic  surfaces with  $p_g=q=1, K^2=4$ and genus 4 nonhyperelliptic Albanese fibrations.

In section 2, we   construct     for the first time   a family  of such surfaces. 
Motivated by the fact that  a canonical    curve of genus 4 is a complete intersection of type $(2,3)$ in $\mathbb{P}^3$, 
   we  take the complete intersection $S$  of a relative hyperquadric $Q$ and a relative hypercubic $X$ in a $\mathbb{P}^3$- bundle  over an elliptic curve $B$.
   We prove that (cf. Theorem \ref{existence}) for general choices of $Q$ and $X$, $S$ is a minimal surface of general type  with $p_g=q=1, K^2=4$ and a genus 4 nonhyperelliptic Albanese fibration. If we only require that $S$ has at most rational double points as singularities, then $S$ is a canonical surface with the required property (cf. Remark \ref{M not empty V1 is decomposable}).

We remark (cf. Remark \ref{M not empty V1 is decomposable}) that for all surfaces in our family, the direct image of canonical sheaf $V_1=f_*\omega_S$  is decomposable. By a result of Catanese-Ciliberto (cf. \cite{CC91} Theorems 1.2 and 1.4), the number of direct summands of $V_1$ is a topological invariant.  Hence   it is also  a deformation invariant. In particular, ``$V_1$ is decomposable" is a deformation invariant   condition.

One would  ask  naturally the following  question: if $S$ is a minimal surface with $p_g=q=1,K^2=4$ and a nonhyperelliptic genus 4 Albanese fibration such that   $V_1$ is decomposable, is $S$ necessarily contained in this family?  

In section 3 and section 4, by studying  the relative canonical algebra of the Albanese fibration, we show that this is true (cf. Theorem \ref{Sigma is contained in the family}).
In particular, we see that  under the assumption that ``$V_1$ is decomposable",   the  canonical model of $S$   is always a complete intersection of a relative hyperquadric and a relative hypercubic in $\mathbb{P}(V_1)$.   

As a consequence, we show that  minimal surfaces   with $p_g=q=1,K^2=4$ and a nonhyperelliptic genus 4 Albanese fibration such that   $V_1$ is decomposable give a 4-dimensional  irreducible subset $\mathcal{M}$ in $\mathcal{M}_{1,1}^{4,4}$. Since ``having a nonhyperelliptic Albanese fibration" is an open condition, the closure of $\mathcal{M}$ is an irreducible component of  $\mathcal{M}_{1,1}^{4,4}$.

 \vspace{2ex}
$\mathbf{Notation~and ~conventions.}$
Throughout this paper we work over the field $\mathbb{C}$ of complex numbers. Unless otherwise stated,  $S$ is  a minimal  surface of general type with $p_g=q=1$.

 We denote by $\Omega_S$  the sheaf of holomorphic 1-forms on $S$,   by    $T_S:=\mathcal{H}om_{\mathcal{O}_S}(\Omega_S,\mathcal{O}_S)$ the tangent sheaf of $S$ and       by $\omega_S:=\wedge^2\Omega_S$    the sheaf of holomorphic 2-forms on $S$. $K_S$ (or simply  $K$ if  no confusion)  is  the canonical divisor of $S$, i.e. $\omega_S\cong \mathcal{O}_S(K_S)$.   $p_g:=h^0(\omega_S), q:=h^0(\Omega_S)$.   The Albanese fibration of $S$ is denoted by  $f: S\rightarrow B:=Alb(S)$. We denote by $g$ the genus of a general fibre of $f$ (which is always called Albanese fibre)  and set $V_n:=f_*\omega_S^{\otimes n}$.

 For an elliptic curve $B$, we denote   by $E_p(r,d)$ ($p$ is a point on $B$) the unique indecomposable rank $r$ vector bundle over $B$ with determinant $\det(E_p(r,d))\cong  \mathcal{O}_B(d\cdot p)$ (see Atiyah   \cite{Ati57}).

 We denote by   $\mathcal{M}_{1,1}^{4,4}$    the Gieseker moduli space of   surfaces of general type  with $p_g=q=1,K^2=4$ and    genus 4 Albanese fibrations.
For divisors, we denote    linear equivalence by  `$\equiv$' and   numerical equivalence  by `$\sim_{num}$'.

\section {Constructing  the family}
In this section, we construct minimal surfaces with $p_g=q=1,K^2=g=4$ and  nonhyperelliptic Albanese fibrations as a complete intersection of a relative hyperquadric and a relative hypercubic in a $\mathbb{P}^3$-bundle over an elliptic curve.

Let $B$ be an elliptic curve. We fix a group structure of $B$,    denote by $0$ the  neutral element and by  $\eta$  a  torsion point of order 3 on $B$ (i.e. $\eta\neq  0,3\eta\equiv 3\cdot 0$).   Let $N:=\mathcal{O}_B(\eta-0)$  and let $V:=E_{[0]}(3,1)\oplus N$ be  the unique indecomposable  rank 3 vector bundle over $B$ with $\det(V)= \mathcal{O}_B(0)$.

Let $\pi: W:=\mathbb{P}(V) \rightarrow B$ be the  $\mathbb{P}^3$-bundle over $B$. Denote by   $T$ be the tautological divisor of $W$ (i.e. $\pi_*\mathcal{O}_W(T)=V$)  and by $H_p$   the fibre of $\pi$ over $p\in B$. The main result in this section is the following

\begin{theorem}
	\label{existence}
Take  a general member $Q \in |2T+H_\eta-H_0|$ and the unique effective divisor  $X \in |3T-H_\tau|$, where $\tau$ is a point on $B$ satisfying  $\tau\equiv  2\cdot 0-\eta$.   Then $S:=Q\cap X$
is a (smooth) minimal  surface with $p_g=q=1,K^2=g=4$ and a nonhyperelliptic Albanese fibration.
\end{theorem}

First we prove that the surface $S=Q\cap X$ in Theorem \ref{existence} is smooth.
 In order to get  global relative coordinates on  $W=\mathbb{P}(V)$, we  need to take  a unramified triple covering.  

\begin{lemma}[ \cite{GOP12} Proposition 3.3]
	There exists an   isogeny $\phi: \tilde{B}\rightarrow B $  such that $\phi^*N\cong \mathcal{O}_{\tilde{B}}$ and  $\phi^*E_{[0]}(3,1) \cong \mathcal{O}_{\tilde{B}}(\tilde{0})\oplus \mathcal{O}_{\tilde{B}}(a)\oplus \mathcal{O}_{\tilde{B}}(b)$, where   $\tilde{0}$ is the neutral element of  $\tilde{B}$ and  $\phi^*(0)=\tilde{0}+a+b$.
\end{lemma}

\vspace{2ex}
Now let $\tilde{V}:=\phi^*(V)=\mathcal{O}_{\tilde{B}}(\tilde{0})\oplus \mathcal{O}_{\tilde{B}}(a)\oplus \mathcal{O}_{\tilde{B}}(b)\oplus \mathcal{O}_{\tilde{B}}$. Then we have the following commutative diagram:
$$\xymatrix{\tilde{W}:=\mathbb{P}(\tilde{V})\ar[rrr]^{\Phi}\ar[d]^{\tilde{\pi}}& &&W=\mathbb{P}(V)\ar[d]^{\pi} \\
	\tilde{B}\ar[rrr]^{\phi} &&&B}
$$
where $\tilde{\pi}:\tilde{W}\rightarrow \tilde{B}$ is the natural projection.
Let $\tilde{T}:=\Phi^*T$ and $\tilde{H}:=\Phi^*H$. The unramified  triple cover $\phi: \tilde{B}\rightarrow B$ induces an automorphism group $G'=<\sigma'>\cong \mathbb{Z}_3$ of $\tilde{B}$; similarly, the unramified  triple cover induces an automorphism group $G=<\sigma>\cong \mathbb{Z}_3$ on $\tilde{W}$. Moreover we have  $G|_{\tilde{B}}=G'$.

\begin{lemma}
	\label{bpf}
The linear system $|2T+H_\eta-H_0|$ on $W$ is base point free.
\end{lemma}
\begin{proof}
	Let $x_1: \mathcal{O}_{\tilde{B}}(\tilde{0}) \rightarrow \tilde{V}, x_2: \mathcal{O}_{\tilde{B}}(a)\rightarrow \tilde{V}, x_3: \mathcal{O}_{\tilde{B}}(b)\rightarrow \tilde{V}, x_4: \mathcal{O}_{\tilde{B}}\rightarrow \tilde{V}$ be the   global relative coordinates on $\tilde{W}$.
Consider the linear system
$$\Delta:=\{f=a_1x_1^2+a_2x_2^2+a_3x_3^2+a_4x_4^2+a_5x_1x_2+a_6x_2x_3+a_7x_3x_1+a_8x_1x_4+a_9x_2x_4+a_{10}x_3x_4=0\}$$
where 
$a_1\in H^0(\mathcal{O}_{\tilde{B}}(2\cdot \tilde{0})\otimes \phi^*\mathcal{O}_B(\eta-0)), a_2=\sigma^*a_1, a_3=\sigma^*a_2, a_4\in H^0(\mathcal{O}_{\tilde{B}}), a_5\in H^0(\mathcal{O}_{\tilde{B}}(\tilde{0}+a)\otimes \phi^*\mathcal{O}_B(\eta-0)), a_6=\sigma^*a_5, a_7=\sigma^*a_6, a_8\in H^0(\mathcal{O}_{\tilde{B}}(\tilde{0})\otimes \phi^*\mathcal{O}_B(\eta-0)), a_9=\sigma^*a_8, a_{10}=\sigma^*a_9$. Note that the action of $G$ is: $x_1\mapsto x_2,x_2\mapsto x_3,x_3\mapsto x_1,x_4\mapsto \zeta x_4, a_4\mapsto \zeta a_4$, where  $\zeta$ is a primitive root of degree 3. Hence
  $\Delta$ is $G$-invariant and it is easy to see that   $\Delta = \Phi^*|2T+H_\eta-H_0|$. Now   it suffices to show that $\Delta$ is base point free.

If $x_4\neq0$, when $a_4$ varies, $\Delta$ has no base point.

 If $x_4=0$, then $a_1x_1^2+a_2x_2^2+a_3x_3^2=0$.  So  we need to show that $\Delta':=\Delta|_{\{x_4=0\}}$ has no base points on $\{x_4=0\}$.  
 Note that $\{x_4=0\}$ is nothing but the  $\mathbb{P}^2$-bundle  $\tilde{\pi} ': \tilde{W}':=\mathbb{P}(\mathcal{O}_{\tilde{B}}(\tilde{0})\oplus \mathcal{O}_{\tilde{B}}(a)\oplus \mathcal{O}_{\tilde{B}}(b)) \rightarrow B$. 
 Let  $\pi ': W':=\mathbb{P}(E_{[0]}(3,1)) \rightarrow B$ be the natural projection  and let  $\Phi':=\Phi|_{\tilde{W}'}$. Then we have the following commutative diagram:
 $$\xymatrix{\tilde{W}'\subset \tilde{W}\ar[rrr]^{\Phi'}\ar[d]^{\tilde{\pi}'}& &&W'\subset{W} \ar[d]^{\pi'} \\
	\tilde{B}\ar[rrr]^{\phi} &&&B}
$$

  Denote by   $T'$ (resp. $H'$)   the tautological divisor  (resp. fibre) of $W'$. Then  we have $T|_{W'}=T', H|_{W'}=H'$ and thus  $\Delta'=(\Phi')^*|2T'+H'_\eta-H'_0|$.  By \cite{CC93} Theorem 1.18,   $|2T'+H'_\eta-H'_0|$ is base point free. Hence   $\Delta'$ is base point free.

 Therefore $\Delta$ is base point free and   consequently 
 $|2T+H_\eta-H_0|$ is   base point free.   
\end{proof}

\begin{lemma}
	\label{singular locus}
$\tilde{X}:=\Phi^*X$ is an integral normal subvariety of $\tilde{W}$, and its singular locus $Sing(\tilde{X})$ is exactly the  curve $C: x_1=x_2=x_3=0$.
\end{lemma}
\begin{proof}
	Note that  $\tilde{X}\in |3\tilde{T}-3\tilde{H}_{\tilde{0}}|$  and  $\tilde{X}$ is defined by an equation of  the form
$g=c_1x_1^3+c_2x_2^3+c_3x_3^3+c_4x_1x_2x_3=0$ (where $c_1,c_2,c_3,c_4 \in H^0(\phi^*N)$) such that $\sigma^*g=g$. Since the action of $\sigma^*$ is: $x_1\mapsto x_2, x_2\mapsto x_3, x_3\mapsto x_1$ and $s\mapsto \zeta s$ for $s\in H^0(\phi^*N)$ ($\zeta$ is a primitive root of degree 3),  we see   $\sigma^*g=\zeta c_1x_2^3+\zeta c_2x_3^3+\zeta c_3x_1^3+\zeta c_4x_1x_2x_3$. Hence  $c_3=\zeta c_2=\zeta^2 c_1, c_4=0$ and     $\tilde{X}$ is defined by  $g=x_1^3+\zeta x_2^3+\zeta^2 x_3^3=0$.

 Since $(\frac{\partial g}{\partial x_1},\frac{\partial g}{\partial x_2},\frac{\partial g}{\partial x_3})=(3x_1^2,3\zeta x_2^2,3\zeta^2x_3^2)$, a point of $\tilde{X}$ is singular if and only if $x_1=x_2=x_3=0$.  Hence the singular locus of $\tilde{X}$ is exactly the curve  $C:=\{x_1=x_2=x_3=0\}$. 

Since $\tilde{X}$ is regular in codimension 1 and   defined by an irreducible equation in $\tilde{W}$,   $\tilde{X}$ is an integral normal subvariety of $\tilde{W}$.
\end{proof}

\begin{cor}
 \label{smooth}
 Let    $S=Q\cap X$ as in Theorem \ref{existence}. Then $S$ is a smooth surface.
\end{cor}
\begin{proof}
	 By Lemma \ref{bpf}, $\Delta$ is base point free.  Note that  a general member $\tilde{Q}:=\Phi^*Q\in \Delta$ does not intersect with $C=Sing(\tilde{X})$.   By Bertini's Theorem, $\tilde{S}=\tilde{Q}\cap \tilde{X}$ is smooth. Since  $\Phi|_{\tilde{S}}: \tilde{S}\rightarrow S$ is a unramified triple cover, we see that   $S=\phi(\tilde{S})$
is also smooth.
 \end{proof}

% Let $C:=\{x_1=x_2=x_3=0\}$, if $a_4\neq 0$, then $\{f=0\} \cap C= \emptyset$, so a general member in $\Delta$ does not intersect with $C$. Thus a general member in $|2T|$ on $W$ does not intersect with $\Phi_*C$. In particular, $|2T|$ has no base point on $\Phi_*C$.

 %Since the natural projection $p: \tilde{W}\rightarrow W':=\mathbb{P}(\mathcal{O}_{\tilde{B}}(a)\oplus \mathcal{O}_{\tilde{B}}(b)\oplus \mathcal{O}_{\tilde{B}}(c))$ ($(y,x_1,x_2,x_3,x_4)\mapsto (y,x_1,x_2,x_3)$ where $y$ is the local coordinate on $\tilde{B}$) is well defined on $\tilde{W}-C$, the projection $p': W\rightarrow \tilde{B}^{(3)}$ is well defined on $W-\Phi_*C$. Denote by $T'$ the tautological bundle of $ \tilde{B}^{(3)}$, then we have $T|_{W-\Phi_*C}=p'^*T'|_{W-\Phi_*C}$. By \cite{CC93} Theorem 1.18, $|2T'|$ is base point free, thus $|2T|$ has no base point on $W-\Phi(C)$.

\vspace{2ex}
Now we calculate the numerical inariants of $S$.
\begin{lemma}
	\label{invariants}
Let $S $ be  as in Theorem \ref{existence}.  Then we have $h^0(\mathcal{O}_S)=p_g(S)=q(S)=1$ and $K_S^2=4$.  Moreover  $S$ is   a minimal surface.
\end{lemma}
\begin{proof}
	Since  $K_W=-4T+H_\eta$, we see $K_S=(K_W+Q+X)|_S=T|_S$ and  $K_S^2=T^2   (3T-H_\tau) (2T+H_\eta-H_0)=4$.

Consider  the cohomology of the exact sequence
$$0\rightarrow \mathcal{O}_W(-T+H_0-H_\eta) \rightarrow \mathcal{O}_W(T) \rightarrow \mathcal{O}_{Q}(T) \rightarrow 0.$$
Since $h^i(\mathcal{O}_W(-T+H_0-H_\eta))=0$ for all $i\geq 0$,
we get $h^0(\mathcal{O}_{Q}(T))=h^0(\mathcal{O}_W(T))=1$ and  $h^i(\mathcal{O}_{Q}(T))=h^i(\mathcal{O}_W(T))=0$ $(i\geq 1)$.

Take the cohomology  of the exact sequence
$$0\rightarrow \mathcal{O}_W(K_W) \rightarrow \mathcal{O}_W(-2T+H_\tau) \rightarrow \mathcal{O}_{Q}(-2T+H_\tau) \rightarrow 0.$$
Since $h^i(\mathcal{O}_W(-2T+H_\tau))=0$ for all $i\geq 0$, $h^i(\mathcal{O}_W(K_W))=h^{4-i}(\mathcal{O}_W)=0$ for $i=1,2$ and  $h^i(\mathcal{O}_W(K_W))=h^{4-i}(\mathcal{O}_W)=1$ for $i=3,4$,
 we get $h^0(\mathcal{O}_{Q}(-2T+H_\tau))=h^1(\mathcal{O}_{Q}(-2T+H_\tau))=0$ and     $h^2(\mathcal{O}_{Q}(-2T+H_\tau))=h^3(\mathcal{O}_{Q}(-2T+H_\tau))=1$.

Now taking the cohomology groups  of the exact sequence
$$0\rightarrow \mathcal{O}_{Q}(-2T+H_\tau) \rightarrow \mathcal{O}_{Q}(T) \rightarrow \mathcal{O}_S(T)\rightarrow 0,$$
we get $h^0(\mathcal{O}_S(T))=h^1(\mathcal{O}_S(T))=h^2(\mathcal{O}_S(T))=1$, i.e. $p_g(S)=q(S)=h^0(\mathcal{O}_S)=1$.
Since $S$ is smooth, $h^0(\mathcal{O}_S)=1$ implies that  $S$ is irreducible.

 Since $|2T+H_\eta-H_0|$ is base point free and  $2T+H_\eta-H_0\sim_{num} 2T$, we see that $T$ is nef. Hence  $K_S=T|_S$ is   nef  and therefore $S$ is minimal.
\end{proof}

\begin{lemma}
	\label{genus is 4}
The restriction map $f:=\pi|_S:  S\rightarrow B$   is the Albanese fibration of $S$, whose general fibres are nonhyperelliptic   of genus 4.
	\end{lemma}
\begin{proof}
	Let $F=H|_S$ be a general fibre of $f$. Then we have  $K_F =(K_S+F)|_F=(T+H)|_F$. Since  $\deg K_F=(T+H) Q X  H=6$, we get $p_a(F)=4$. Since $f$ is a morphism from a smooth complex surface to a smooth complex curve, by the generic smoothness,     $F$   is smooth.

	Since $H^1(T+H)=0$, we see $|(T+2H)||_H\cong |\mathcal{O}_H(1)|$.  In particular,   the relative canonical map of $f$ is birational,  hence  $F$ is nonhyperelliptic.

	Now we show that $F$ is connected.
	 If $F$ is not connected, by Stein factorization, $f$ can be factored as   $S\stackrel{f'}{\rightarrow} B'\stackrel{\upsilon}{\rightarrow} B$, where $f': S\rightarrow B'$ has connected fibres and $\upsilon: B'\rightarrow B$ is finite of degree $m\geq 2$.    Then we have  $F =\sum_{i=1}^m \upsilon^*\Gamma_i\sim_{num}m \upsilon^*\Gamma$, where   $\Gamma_i$'s  and $\Gamma$ are  fibres of $f'$.  Since $K_S F=6=mK_S  \upsilon^*\Gamma$ and $K_S \upsilon^*\Gamma$ is even, we get $m=3$ and $K_S \upsilon^*\Gamma=2$. Hence  
	 $\upsilon^*\Gamma$ is of genus 2 and     the relative canonical map of $f$ cannot be birational, a contradiction.
	Therefore $F$ is a  connected. Since $F$ is smooth, it must be   irreducible.

By the universal property of the Albanese map, we see that  $f$ is the Albanese fibration of $S$.
\end{proof}
	
	 Combining Corollary \ref{smooth}, Lemma \ref{invariants} and Lemma \ref{genus is 4}, we get  Theorem \ref{existence}.

  \begin{remark}   \label{M not empty V1 is decomposable}
  (i) In fact, $K_S=T|_S$ ( in  Lemma \ref{invariants}) is ample: note that $T$ is nef and  the only irreducible curve in $W$ whose  intersection number  with $T$ is zero is the curve $\Phi(C)$, where $C=\{x_1=x_2=x_3=0\}$ is defined   in Lemma \ref{singular locus}. Since $S\cap C=\emptyset$, we see that $T|_S$ is ample.   
   In particular, $S'$ has no $(-1)$ or $(-2)$ curves and it is  a canonical surface.   This is true for any choices of $Q\in |2T+H_\eta-H_0|$ such that $Q\cap X$ has at most rational double points as singularities.

  (ii) From Lemma \ref{genus is 4}, it is easy  to see $V_1:=f_*\omega_S=V$. In particular, $V_1$ is decomposable. 

\end{remark}

 \begin{defn} \label{defn of family M}
 Notation $B,V,W,T,H,\eta,\tau$ are as before.

    We define  $M'$ the family of (canonical) surfaces $S'$ obtained  as a complete intersection   $S'=Q\cap X$ with $Q\in |2T+H_\eta-H_0|$ and $X\in |3T-H_\tau|$ such that $S'$ has at most rational double points as singularities.

    We define $M$ the family of (minimal) surfaces $S$ such that $S$ is the minimal resolution of a surface $S'\in M'$.

    We denote by $\mathcal{M}$ the image of $M'$ in $\mathcal{M}_{1,1}^{4,4}$.  By Theorem \ref{existence} $\mathcal{M}$ is not empty.
 \end{defn}

	 \begin{lemma}
	 	\label{dimesion of M}
$\mathcal{M}$ is a 4-dimensional irreducible subset of $\mathcal{M}_{1,1}^{4,4}$.
		 	 \end{lemma}
 \begin{proof}    As in \cite{Pol09} Proposition 5.2, we first compute the dimension of the corresponding  parameter space, and then subtract the dimension of the general fibre of $M\rightarrow \mathcal{M}\subset \mathcal{M}_{1,1}^{4,4}$.

 There is one parameter for the elliptic curve $B$.  Since $V=E_{[0]}(3,1)\oplus N$ is fixed, $\mathbb{P}(V)$ is determined by $B$. Since we have fixed the neutral element of $B$, only finitely many automorphisms of $B$ acting on our data.  Similarly, for the automorphisms of $\mathbb{P}(V)$, we only need to consider the automorphisms of $V$.

Since   $\dim |3T-H_\tau|=0$, $\dim |2T+H_\eta-H_0|=5$ and  $\dim\mathbb{P}(Aut(V))=2$,  we have
   	 $$\dim\mathcal{M}=1+\dim|2T+H_\eta-H_0|+\dim|3T-H_\tau|-\dim \mathbb{P}(Aut(V))=1+5+0-2=4.$$
  \end{proof}

	 \begin{remark}
	 	\label{examples with pg is 2}
In Theorem \ref{existence}, if we take different $V$, we would get some other  examples.

(1) 	If we take $\eta=0$, i.e. $N\cong \mathcal{O}_B$, then $\tau=0$, $Q\in |2T|$ and  $X \in |2T|$.  Since $h^0(\mathcal{O}_W(T))=2$ and  $h^1(\mathcal{O}_W(T))=1$,   we   get a family of  minimal surfaces     with   $K_S^2=4, p_g(S)=q(S)=2$ and genus 4 nonhyperelliptic fibrations (over an elliptic curve).  By the same argument, we  see the image of this family in a 4-dimensional irreducible subset in  $\mathcal{M}_{2,2}^4$, the moduli space of minimal surfaces with $p_g=q=2,K^2=4$.

% {\color{red} However, since $q(S)=2$ and $f:S\rightarrow B$ is a fibration over an elliptic curve, $f$ is not the Albanese map of $S$. In fact, by Stein factorization, one can easily prove that the image of the Albanese map of $S$ is of dimension 2, i.e. its Albanese map $S\rightarrow Alb(S)$ is surjective.

% \vspace{3ex} 
 % (2) If we take $V=E_{[0]}(3,2)\oplus N$ with $N$ as in Theorem \ref{existence},  $Q\in |2T+H_\eta-H_0|$ and $X\in |3T-H_\tau|$ ($\tau\equiv 2\cdot 0-\eta$).  By a similar argument as  before, one can show that a general surface $S=Q\cap X$ is irreducible  smooth minimal surface. In this case,  we have  $K_S=(T+H_0)|_S$ and hence $K_S^2=(T+H_0)^2(2T+H_\eta-H_0)(3T-H_\tau)=22$.

% From the exact sequences
% $$0\rightarrow \mathcal{O}_{Q}(-2T+H_\tau+H_0) \rightarrow \mathcal{O}_{Q}(T+H_0) \rightarrow \mathcal{O}_S(T+H_0)\rightarrow 0$$
 % $$0\rightarrow \mathcal{O}_W(K_W) \rightarrow \mathcal{O}_W(-2T+H_\tau+H_0) \rightarrow \mathcal{O}_{Q}(-2T+H_\tau+H_0) \rightarrow 0.$$
  % $$0\rightarrow \mathcal{O}_W(-T+h_\tau) \rightarrow \mathcal{O}_W(T+H_0) \rightarrow \mathcal{O}_{Q}(T+H_0) \rightarrow 0.$$
%  and by an argument as before, one  can show that $p_g(S)=h^0(\mathcal{O}_S(T+H_0))=5, q=h^1(\mathcal{O}_S(T+H_0))=1, h^0(\mathcal{O}_S)=h^2(\mathcal{O}_S(T+H_0))=1.$

 %Since $q(S)=1$, $f:S\rightarrow B$ is the Albanese fibration of $S$. So we get a 17-dimensional family ($1+10+9-3$) of minimal surfaces with $p_g=5,q=1,K^2=22$ and genus 4 Albanese fibrations. Moreover $f_*\omega_S=V$. }

 (ii) If we take $V=E_{[0]}(3,d)\oplus N$ ($d\geq 1$) with $N$ as in Theorem \ref{existence},  $Q\in |2T+H_\eta-H_0|$ and $X\in |3T-((d-2)H_0+2H_\eta)|$. Then  we have $K_S=T|_S$ and thus $K^2_S=T^2(2T+H_\eta-H_0)(3T-((d-2)H_0+2H_\eta))=4d$. From the exact sequences
  $$0\rightarrow \mathcal{O}_{Q}(-2T+(d-2)H_0+2H_\eta) \rightarrow \mathcal{O}_{Q}(T) \rightarrow \mathcal{O}_S(T)\rightarrow 0,$$
   $$0\rightarrow \mathcal{O}_W(K_W) \rightarrow \mathcal{O}_W(-2T+(d-2)H_0+2H_\eta) \rightarrow \mathcal{O}_{Q}(-2T+(d-2)H_0+2H_\eta) \rightarrow 0,$$
   $$0\rightarrow \mathcal{O}_W(-T+H_0-H_\eta) \rightarrow \mathcal{O}_W(T) \rightarrow \mathcal{O}_{Q}(T) \rightarrow 0,$$

   we get  $p_g(S)=h^0(\mathcal{O}_S(T))=d, q=h^1(\mathcal{O}_S(T))=1, h^0(\mathcal{O}_S)=h^2(\mathcal{O}_S(T))=1.$  By a similar argument as before,  one can show that a general $S$ is a minimal surface.

   Since $q(S)=1$, $f:S\rightarrow B$ is the Albaense fibration of $S$. Hence we get a series of  families  of minimal surfaces with $p_g=d,q=1,K^2=4d=4\chi$ and genus 4 nonhyperelliptic Albanese fibrations. Moreover, $f_*\omega_S=V$.

%{\color{red}  Similarly, if we let $d$ vary in (i), we  we get a series of  families  of minimal surfaces with $p_g=d+1,q=2,K^2=4d=4\chi$ and genus 4 nonhyperelliptic   fibrations (over an elliptic curve).      What about (ii) when $d$ varies ? }
  
   \end{remark}

In the following sections, we prove that the converse of Theorem \ref{existence} is also true, i.e. the canonical model of  a minimal surface with $p_g=q=1,K^2=g=4$ and a  nonhyperelliptic  Albanese fibration such that $V_1=f_*\omega_S$ is   a complete intersection  of the form $S=Q\cap X$ as in Theorem \ref{existence}.

 The proof depends heavily on studying  the structure of the relative canonical algebra of the Albanese fibration of $S$.    First we    recall   some generalities  on  nonhyperelliptic fibrations of genus 4 that we shall use later. (see Catanese-Pignatelli \cite{CP00} or Takahashi \cite{Tak12})

\section{Some generalities on  nonhyperelliptic fibrations of genus 4}	
  Let  	 $f: S\rightarrow B$  be a relatively minimal   nonhyperelliptic fibration of genus $g=4$. Let  $b=g(B)$  be the genus of $B$.  Define $K_{S/B}:=K_S-f^*K_B$ and   $\omega_{S/B}:=\mathcal{O}_S(K_{S/B})$. Then  $K_{S/B}^2=K_S^2-8(b-1)(g-1)$. Define further  $\chi(S/B):=\chi(\mathcal{O}_S)-(g-1)(b-1)$.

    Let   $V_n:=f_*\omega_{S/B}^{\otimes n}$  and let   $\mathcal{R}(f):=\bigoplus_{n\geq 0}V_n$ be the relative canonical algebra of $f$. 	Denote by $\mu_{n,m}: V_n\otimes V_m\rightarrow V_{n+m}$, respectively by $\sigma_n: S^n(V_1):= \sym^n(V_1)\rightarrow V_n$   the homomorphism induced by multiplication. Set   $\mathcal{L}_n:= \kernel \sigma_n$ and $\mathcal{T}_n:= \cokernel \sigma_n$.

        Since we have assumed that the general fibre of $f$ is nonhyperelliptic, by Max Noether's theorem on canonical curves, $\mathcal{T}_n$ is a torsion sheaf.  Hence we have $\rank \mathcal{L}_n=\rank S^n(V_1)-\rank V_n$ and $\deg \mathcal{L}_n=\deg S^n(V_1)-\deg V_n+\deg \mathcal{T}_n$.   By  \cite{CP06} sections 2 and 3 we have   $\rank S^n(V_1)=\binom{n+3}{n}$, $\deg S^n(V_1)=\binom{n+3}{4}\chi(S/B)$;  $\rank V_n=3(2n-1)$, $\deg V_n= \frac{1}{2}n(n-1)K_{S/B}^2+\chi(S/B)$.   In particular, we have $\rank \mathcal{L}_2=1$ and   $\rank \mathcal{L}_3=5$.

    \vspace{2ex}
    Now let $\pi: \mathbb{P}(V_1)\rightarrow B$ be the natural projective bundle over $B$ and let $w: S\dashrightarrow  \mathbb{P}(V_1)$ be the relative canonical (rational) map of $f$. Let $\Sigma:=w(S)$ be the relative canonical image of $S$.  Now we give some  relations between the  structure of the relative canonical algebra $\mathcal{R}(f)$  and the geometry of the  relative canonical image $\Sigma$.

           First we show that there is a (unique) relative hyperquadric $Q\subset \mathbb{P}(V_1)$ that contains $\Sigma$.

        \subsection{The relative hyperquadric $Q$ containing $\Sigma$}
        Since $\rank \mathcal{L}_2=1$, $\mathcal{L}_2$ is a line bundle. The injection $\mathcal{L}_2\hookrightarrow S^2(V_1)$ defines a section $q\in H^0(B,S^2(V_1)\otimes \mathcal{L}_2^{-1})\cong H^0(\mathcal{O}_{\mathbb{P}(V_1)}(2)\otimes \pi^*\mathcal{L}_2^{-1})$, and thus defines a relative hyperquadric $Q\subset \mathbb{P}(V_1)$. Note that $Q$ contains $\Sigma$: if we restrict to a general fibre $H\cong \mathbb{P}^3$ of $\pi$, $Q|_H$ is exactly the unique hyperquadric in $\mathbb{P}^3$ that contains $\Sigma|_{H}$ (since for general $H$, $\Sigma|_{H}$ is a smooth  nonhyperelliptic curve of genus 4), hence $Q$ contains a Zariski open subset of $\Sigma$. Since $Q$ is a closed subvariety of $\mathbb{P}(V_1)$, we have $\Sigma\subset Q$.

When $\mathcal{T}_2=0$, we have the following lemma due to Takahashi and Konno (cf. \cite{Tak12} Lemma 1).

 \begin{lemma}  \label{Takahashi's lemma generated in degree 1}
If $\mathcal{T}_2=0$, then   $w: S\rightarrow \mathbb{P}(V_1)$ is a morphism and    $\mathcal{R}(f)$ is generated in degree 1. Hence
 $\Sigma$ is isomorphic to its relative canonical model.  In particular, $\Sigma$ has at most rational double points as singularities.
 \end{lemma}

 \begin{proof}[Sketch of Takahashi's proof]

(i) $|2K_S+f^*\delta|$  is base point free for a sufficiently ample divisor $\delta$ on $B$, hence  $\mathcal{T}_2=0$ implies that $|K_F|$ is base point free for any fibre $F$ of $f$.

(ii) Now we use  a result of Konno (see \cite{Kon01} Corollary 1.2.3, let $i=0$): assume that  $D$ is a curve  contained in a smooth surface  which satisfies $H^0(D, -mK_D)=0$ for any positive integer  $m$ and $h^0(D,\mathcal{O}_D)=1$. Set $g(D):=h^0(D,K_D)$.  If  $K_D$ is generated by its global sections, then the maps
$$H^0(D,K_D)\otimes H^0(D, (j-1)K_D)\rightarrow H^0(D,jK_D)$$
are surjective for $j\geq 4$, and also for $j=3$ when  $g(D)\geq 3$.

Here let $D=F$. Then we have  $g(F)=4$ and  all conditions above are satisfied for $F$.  Hence  $H^0(F,K_F)\otimes H^0(F, (j-1)K_F)\rightarrow H^0(F,jK_F)$  is surjective for $j\geq 3$.
Since $F$ is arbitrary, we see that $\sigma_j$ is surjective for all $j\geq 3$. Since   moreover we have assumed that $\sigma_2$ is surjective, we see that $\mathcal{R}(f)$  is generated in degree 1 and consequently $\Sigma$ is  isomorphic to its  relative canonical model.
\end{proof}

\begin{remark}   \label{g is begger than 3 R(f) generated in degree 2}

(i) The above proof works for  all (nonhyperelliptic) fibrations of genus $g\geq 3$:

 Let $S$ be a minimal surface of general type (we   need this condition to ensure that $|2K_S+f^*\delta|$  is base point free for a sufficiently ample divisor $\delta$ on $B$) and let  $f:S\rightarrow B$ be a   fibration of genus $g\geq 3$. If $\sigma_2: S^2(V_1)\rightarrow V_2$ is surjective,  then   $\mathcal{R}(f)$  is generated in degree 1 and $\Sigma$ is isomorphic to its  relative canonical model.

(ii) Let $S$ be a minimal  surface  of general type endowed  with a fibration $f: S\rightarrow B$.   Then its canonical model coincides with its relative canonical model except in the case where $B\cong \mathbb{P}^1$ and $f$ has a $(-2)$ curve mapped surjectively to $B$.
 \end{remark}

\subsection{Describing $\Sigma$ as a divisor on $Q$}

         Recall that we have the following two exact sequences induced by multiplication
         \begin{equation}  \label{exact sequence sigma 2}
            0\rightarrow \mathcal{L}_2 \stackrel{\beta_2} {\rightarrow} S^2(V_1)\stackrel{\sigma_2}{\rightarrow} V_2\rightarrow \mathcal{T}_2\rightarrow 0
         \end{equation}

          \begin{equation}  \label{exact sequence sigma 3}
            0\rightarrow \mathcal{L}_3 \stackrel{\beta_3} {\rightarrow}  S^3(V_1)\stackrel{\sigma_3}{\rightarrow} V_3\rightarrow \mathcal{T}_3\rightarrow 0
         \end{equation}

       Tensoring the exact sequence (\ref{exact sequence sigma 2})  with $V_1$  and using  the definitions of $\mu_{m,n}$ $\sigma_n$, we get the following commutative diagram
         $$\xymatrix    {0 \ar[r] & \mathcal{L}_2\otimes V_1 \ar[r]^{\beta_2\otimes id}  & S^2(V_1)\otimes V_1 \ar[r]^{\sigma_2\otimes id} \ar[d]^{\gamma} &V_2\otimes V_1 \ar[r] \ar[d]^{\mu_{2,1}} &\mathcal{T}_2\otimes V_1 \ar[r] &0 \\
         0 \ar[r] & \mathcal{L}_3 \ar[r]^{\beta_3} & S^3(V_1) \ar[r]^{\sigma_3} &V_3 \ar[r] & \mathcal{T}_3 \ar[r] & 0
         }$$
         where $\gamma$ is the natural surjection in $\sym(V_1)$.
           Set    $h:=\gamma \circ (\beta_2\otimes id): \mathcal{L}_2\otimes V_1\rightarrow S^3(V_1)$. Since $\sigma_3\circ h(\mathcal{L}_2\otimes V_1)=\mu_{2,1}\circ (\sigma_2\otimes id)\circ(\beta_2\otimes id)(\mathcal{L}_2\otimes V_1)=0$, we see  $h(\mathcal{L}_2\otimes V_1)\subset \beta_3(\mathcal{L}_3)$. Hence the map $h$ factors through $\mathcal{L}_3$, i.e. there is a (unique) map $j: \mathcal{L}_2\otimes V_1\rightarrow \mathcal{L}_3$ such that $\beta_3\circ j=h$.

      \begin{lemma}
      \label{injection  from L2 to L3}
          $j$ and $h$ are   injections.
       \end{lemma}
      \begin{proof}
            First we show that $j$ is generically injective.  Taking a general point $p\in B$,  the  fibre $F_p:=f^{-1}(p)$   is a smooth nonhyperelliptic curve of genus 4, and its canonical image $F'_p$ is a complete intersection of type $(2,3)$ in $\mathbb{P}^3$. Let $I$ be the ideal sheaf of $F_p'$ in $\mathbb{P}^3$ and set $I_n:=I(n)$.
            Then   the  restriction of   $j$ to   the point  $p$  is exactly the map  $I_2\otimes H^0(F_p',\mathcal{O}_{F'_p}(1))\rightarrow I_3$, which is injective. Hence  $j$ is generically injective, and $\kernel (j)$  is a torsion sheaf.

            Since moreover $\kernel(j)$ is a subsheaf of the locally free sheaf $\mathcal{L}_2\otimes V_1$, it must be 0. Therefore $j$ is injective.
 Since $h=\beta_3\circ j$ and $\beta_3$ and $j$ are both injections, $h$    is also an injection.
            \end{proof}
         Taking the quotient of (\ref{exact sequence sigma 3})   by $\mathcal{L}_2\otimes V_1$, we get the following exact sequence

         \begin{equation}  \label{exact sequence sigma 3'}
               0\rightarrow \mathcal{L}_3/(\mathcal{L}_2\otimes V_1) \stackrel{\beta_3} {\rightarrow}  S^3(V_1)/(\mathcal{L}_2\otimes V_1) \stackrel{\sigma_3}{\rightarrow} V_3\rightarrow \mathcal{T}_3\rightarrow 0
         \end{equation}

         \begin{lemma} \label{coherent sheaf on a curve splitting}
                   Let $\mathcal{F}$ be a coherent sheaf on a smooth irreducible curve $C$. Then we can write $\mathcal{F}=\mathcal{F}_l\oplus \mathcal{F}_t$, where $\mathcal{F}_l=\mathcal{F}^{\vee \vee}$  is a locally free sheaf and $\mathcal{F}_t$   is a torsion sheaf on $C$.
         \end{lemma}
            \begin{proof}
              Let $\mathcal{F}^{\vee}:=\mathcal{H}om(\mathcal{F}, \mathcal{O}_C)$ be the dual sheaf of $\mathcal{F}$. Then $\mathcal{F}^{\vee \vee}$ is a locally free sheaf and there is a natural map $u: \mathcal{F}\rightarrow \mathcal{F}^{\vee \vee}$.

               Let $p$ be a point on $C$ and restrict $u$ to $p$, we get  $u|_p: \mathcal{F}_p\rightarrow \mathcal{F}_p^{\vee \vee}$.  Since $C$ is a smooth irreducible curve, $(\mathcal{O}_C)_p$ is a principle ideal domain. Hence $\mathcal{F}_p$ can be written as a direct sum of free modules and torsion modules, and the direct sum of such free modules are nothing but $\mathcal{F}_p^{\vee \vee}$.
               In particular, $u|_p$ is an surjection  at each point $p\in C$ and $u|_p$ is an isomorphism for a general point $p\in C$. Hence $u$ is surjective and $\kernel(u)$ is a torsion sheaf.

               Set  $\mathcal{F}_t:=\kernel(u)$. Then we have the following short exact sequence
               $$0\rightarrow \mathcal{F}_t \rightarrow \mathcal{F}\rightarrow \mathcal{F}^{\vee \vee} \rightarrow 0.$$
               Since $\ext^1(\mathcal{F}^{\vee \vee}, \mathcal{F}_t)\cong H^1(C,\mathcal{F}_t \otimes \mathcal{F}^{\vee})=0$ (the support of $\mathcal{F}_t \otimes \mathcal{F}^{\vee}$ is of dimension 0), we have $\mathcal{F}=\mathcal{F}^{\vee \vee}\oplus \mathcal{F}_t$.
               \end{proof}

   Now we  write $\mathcal{L}_3/(\mathcal{L}_2\otimes V_1)=\mathcal{L}_3'\oplus \mathcal{T}$ with $\mathcal{T}$ its torsion part, and     write $S^3(V_1)/(\mathcal{L}_2\otimes V_1)=\tilde{V}_3\oplus \mathcal{T}'$ with $\mathcal{T}'$ its torsion part.    From the exact sequence     (\ref{exact sequence sigma 3'}), it is easy to see that $\mathcal{T}=\mathcal{T}'$. Hence we get the following exact sequence

   \begin{equation}  \label{exact sequence sigma 3' modulo torsion}
               0\rightarrow \mathcal{L}'_3 \stackrel{\beta_3} {\rightarrow}   \tilde{V}_3 \stackrel{\sigma_3}{\rightarrow} V_3\rightarrow \mathcal{T}_3\rightarrow 0.
         \end{equation}
        The injection $\mathcal{L}'_3 \hookrightarrow   \tilde{V}_3$   defines a section in $H^0(B, \tilde{V}_3\otimes (\mathcal{L}_3')^{-1})\cong H^0(Q, \mathcal{O}_Q(3)\otimes \pi_Q^*(\mathcal{L}_3')^{-1})$ (here $\pi_Q:=\pi|_Q$), hence defines a divisor $Y$ on $Q$. As in section 3.1, it is easy to see that $Y$ contains an open subset of $\Sigma$, hence $Y$ contains $\Sigma$.

      \begin{remark}    \label{Y equas Sigma}
      As in  \cite{CP00} Theorem 2.5, we  have   $Y=\Sigma$:

        Let $H$ be a fibre of $\pi:\mathbb{P}(V_1)\rightarrow B$. If $Y\cap H$ is a curve, it has the same degree  as $\Sigma\cap H$ and it contains  $\Sigma\cap H$, hence they coincide. So, if $Y\neq \Sigma$, then there exists a fibre $H$ such that $H\cap Y$ is a component of $H\cap Q$.

        If $Q\cap H=Y\cap H$, we get some torsion element in the cokernel of $\mathcal{L}_3\rightarrow S^3(V_1)$, which is a subsheaf of the locally free sheaf $V_3$, a contradiction. If  $Q\cap H\neq Y\cap H$, then  $Q\cap H$ is the union of two plans and $Y\cap H$  is one of them. Since $Y\cap H$ contains $\Sigma\cap H$  and $\Sigma\cap H$ is a nondegenerate curve in $H\cong \mathbb{P}^3$, we get a contradiction.  Therefore we have  $Y=\Sigma$.
       \end{remark}

\section{The case $p_g=q=1$ and $K^2=4$}
  Now we apply the general theories in section 3 to our case. We use  the same notation as  in section 3. Due to technical reasons, we assume further that $V_1=f_*\omega_S(=f_*\omega_{S/B})$ is decomposable.
The main theorem in this section is  the following

   \begin{theorem} \label{Sigma is contained in the family}
    Let $S$ be a   minimal surface  with $p_g=q=1,K^2=4$ and genus 4 nonhyperelliptic Albanese fibration $f: S\rightarrow B=Alb(S)$
     such  that $V_1=f_*\omega_S$ is decomposable. Then  $S$ is isomorphic to a surface in the family $M$  (cf. Definition \ref{defn of family M}).
 \end{theorem}

    From now on, we always assume the following condition:

$(\star)$   $S$ is a minimal surface with $p_g=q=1,K^2=4$ and a genus 4 nonhyperelliptic Albanese fibration $f: S\rightarrow B=Alb(S)$ such that $V_1=f_*\omega_S$ is decomposable.

\vspace{2ex}
Since  $\deg V_1=1$ and  $\rank V_1=g$,   $V_1$ has a unique decomposition into indecomposable vector bundles $V_1=\bigoplus _{i=1}^kW_i$ with $\deg W_1=1$ and    $\deg W_i=0, H^0(W_i)=0$ $(2\leq i\leq k)$ (cf. \cite{CC91} p. 56).

   \begin{lemma} \label{BZ00 iota}
   Under   assumption $(\star)$ and  up to a translation of $B$, we can assume  $V_1=E_{[0]}(3,1)\oplus N$ with $N$ a nontrivial torsion line bundle   over $B$.
\end{lemma} 	
\begin{proof}
	Since $K_{S/B}^2=K_S^2=4$ and $\chi(S/B)=\chi(\mathcal{O}_S)-(g-1)(g(B)-1)=1$, we see  $\frac{K_{S/B}^2}{\chi(S/B)}=4$. By \cite{BZ00} Theorem 2, the number $k$ of direct summands of $V_1$ is   at most 2. Since we have assumed that $V_1$ is decomposable, we have $k=2$. Hence up to a translation of $B$, we can write
$V_1=E_{[0]}(3,1)\oplus N$ with $N$ a degree 0 line bundle over $B$. By \cite{CP06} Remark 2.10 and $h^0(N)=0$, we see that $N$ is a nontrivial torsion line bundle.
 \end{proof}

 \begin{lemma}    \label{T2 is 0}
  Under   assumption $(\star)$, we have $\deg \mathcal{L}_2=0$ and $\mathcal{T}_2=0$. Hence we have the following short exact sequence
    \begin{equation}  \label{short exact sequence sigma 2}
            0\rightarrow \mathcal{L}_2 \stackrel{\beta_2} {\rightarrow} S^2(V_1)\stackrel{\sigma_2}{\rightarrow} V_2\rightarrow 0
         \end{equation}
         Moreover we have $\mathcal{L}_2\cong N^{\otimes 2}$.
 \end{lemma}
  \begin{proof}
  Recall that $\mathcal{L}_2$ is a line bundle with $\deg \mathcal{L}_2=\deg S^2(V_1)-\deg V_2+\deg \mathcal{T}_2=\deg \mathcal{T}_2$.
  Since $\beta_2:\mathcal{L}_2\rightarrow S^2(V_1)$ is an  injection, we have (see   Appendix Lemma \ref{injection  less than maxslope}) $\slope(\mathcal{L}_2)\leq \maxslope(S^2(V_1))=\frac{2}{3}$, where $\maxslope(S^2(V_1)):=\max\{\slope(W)| W$ is  a direct summand of $S^2(V_1)\}$. Since moreover $\deg \mathcal{T}_2\geq 0$,  we get $\deg \mathcal{L}_2=\deg \mathcal{T}_2=0$.

  Twisting the exact sequence (\ref{short exact sequence sigma 2}) by $N^{-2}$ and taking cohomology, we get
  $$H^1(\mathcal{L}_2\otimes N^{-2}) \stackrel{\lambda}{\rightarrow} H^1(S^2(V_1)\otimes N^{-2})\rightarrow H^1(V_2\otimes N^{-2})\rightarrow 0.$$   By Riemann-Roch on $B$, we have
    $$\chi(V_2\otimes N^{-2})=h^0(V_2\otimes N^{-2})-h^1(V_2\otimes N^{-2})=h^0(\mathcal{O}_S(2K)\otimes f^*N^{-2})-  h^1(V_2\otimes N^{-2}).$$
    By Riemann-Roch   on $S$, we have
    $$\chi(\mathcal{O}_S(2K)\otimes f^*N^{-2})= h^0(\mathcal{O}_S(2K)\otimes f^*N^{-2})=\chi(\mathcal{O}_S)+K_S^2=5.$$
     Hence we get $h^1(V_2\otimes N^{-2})=0$ and consequently the map $\lambda$ is surjective. Since
     $$S^2(V_1)\otimes N^{-2}= S^2(E_{[0]}(3,1))\otimes N^{-2}\oplus E_{[0]}(3,1)\otimes N^{-2}\oplus \mathcal{O}_B,$$

     by Lemma \ref{Atiyah lemma 15 and theorem 5}, we see $h^0(S^2(V_1)\otimes N^{-2})=h^0(\mathcal{O}_B)=1$. Hence $h^1(\mathcal{L}_2\otimes N^{-2})\geq 1$. Since moreover $\rank \mathcal{L}_2\otimes N^{-2}=1$ and $\deg \mathcal{L}_2\otimes N^{-2}=0$, we have $\mathcal{L}_2\otimes N^{-2}\cong \mathcal{O}_B$, i.e. $\mathcal{L}_2\cong N^{\otimes 2}$.
  \end{proof}

The following corollary follows easily from Lemmas \ref{Takahashi's lemma generated in degree 1} and \ref{T2 is 0}.
\begin{cor}  \label{Sigma is canonical model}
 Under condition $(\star)$, $\mathcal{R}(f)$ is generated in degree 1 and $\Sigma$ is isomorphic to its canonical model. In particular, $\Sigma$ has at most rational double points as singularities.
\end{cor}

\begin{remark}
\label{T2 is 0 for genus bigger than 4}
(i)   The proof of Lemma \ref{T2 is 0}  can be applied to a more general case:

 Let $S$ be a minimal surface with $p_g=q=1,K^2=4$ and a nonhyperelliptic Albanese fibration $f: S\rightarrow B=Alb(S)$  of genus $g\geq 4$. Then we have $\mathcal{T}_2=0$, i.e. $\sigma_2: S^2(V_1)\rightarrow V_2$ is   surjective.
 \begin{proof}
   By Lemma \ref{BZ00 iota}, we see that $V_1$ is either indecomposable or is a direct sum of two indecomposable vector bundles, hence $\maxslope(S^2(V_1))=\frac{2}{g}$ or $\frac{2}{g-1}$. In particular, we always have  $\maxslope(S^2(V_1))\leq \frac{2}{g-1}$. On the other hand, we have $\rank \mathcal{L}_2=\rank S^2(V_1)-\rank V_2=\frac{(g-2)(g-3)}{2}$ and $\deg \mathcal{L}_2=\deg S^2(V_1)-\deg V_2=\deg \mathcal{T}_2+g-4$. Since $\beta_2$ is injective, we have $\slope(\mathcal{L}_2)=\frac{2(\deg \mathcal{T}_2+g-4)}{(g-2)(g-3)}\leq \maxslope(S^2(V_1))\leq \frac{2}{(g-1)}$. Therefore we get $\deg \mathcal{T}_2=0$, i.e. $\mathcal{T}_2=0$.
       \end{proof}

 (ii)  Using a similar method, Catanese and Konno (cf. \cite{Kon93} Lemma 2.5) proved that: for minimal surfaces    with $p_g=q=1,K^2=4$ and  nonhyperelliptic Albanese fibrations of genus $g$, if $V_1$ is semi-stable (which in our case is equivalent to ``$V_1$ is indecomposable"  by Lemma \ref{indecomposable is semistable}), then $g\leq 6$.

 The proof is easy: since $V_1$ is semi-stable, $S^2(V_1)$ is semi-stable. Hence we have $\slope(S^2(V_1))=\frac{2}{g}\leq slope(V_2)=\frac{K^2+1}{3(g-1)}$ and consequently $g\leq \frac{6}{5-K^2}=6$.   Unluckily, this proof does not work if $V_1$ is not semi-stable or $K^2\geq 5$.

 (iii)  The exact sequence (\ref{short exact sequence sigma 2}) splits and we have $V_2=S^2(E_{[0]}(3,2))\oplus E_{[0]}(3,1)\otimes N$:

 From the proof of Lemma \ref{T2 is 0}, we see that the map $\lambda: H^1(\mathcal{L}_2\otimes N^{-2})\rightarrow H^1(N^{\otimes 2}\otimes N^{-2})$ is nonzero. Hence the composition map $pr\circ \beta_2: \mathcal{L}_2\rightarrow S^2(V_1)\rightarrow N^{\otimes 2}$ (where $pr: S^2(V_1)\rightarrow N^{\otimes 2}$ is the natural projection) is nonzero. Since moreover $\mathcal{L}_2\cong N^{\otimes 2}$ is a line bundle,  $pr\circ \beta_2$ is an isomorphism. Therefore the exact sequence (\ref{short exact sequence sigma 2}) splits and  $V_2=S^2(E_{[0]}(3,2))\oplus E_{[0]}(3,1)\otimes N$.

\end{remark}

    By Lemma \ref{injection  from L2 to L3}, we have two natural injections $j: \mathcal{L}_2\otimes V_1\rightarrow \mathcal{L}_3$ and $h: \mathcal{L}_2\otimes V_1\rightarrow S^3(V_1)$. Hence we have the following two exact sequences
    \begin{equation} \label{exact sequence j}
                 0\rightarrow  \mathcal{L}_2\otimes V_1\rightarrow \mathcal{L}_3 \rightarrow \cokernel(j) \rightarrow 0
    \end{equation}
      \begin{equation} \label{exact sequence h}
                 0\rightarrow  \mathcal{L}_2\otimes V_1\rightarrow S^3(V_1) \rightarrow \cokernel(h) \rightarrow 0
    \end{equation}
     Now we show
     \begin{lemma}   \label{exact sequence j h split}
             Under assumption $(\star)$,  the exact sequences (\ref{exact sequence j}) and (\ref{exact sequence h}) are splitting. In particular,   $\cokernel(j)$   and $\cokernel(h)$ are locally free sheaves.
     \end{lemma}
   \begin{proof}
      It suffices to show that (\ref{exact sequence h}) is splitting: if (\ref{exact sequence h}) is splitting, then we have a map $h': S^3(V_1)\rightarrow \mathcal{L}_2\otimes V_1$ such that $h'\circ h=id$. Let $j':=h'\circ \beta_3:\mathcal{L}_3\rightarrow \mathcal{L}_2\otimes V_1$.  Then we have $j' \circ j=h'\circ \beta_3\circ j=h'\circ h=id$. Hence  (\ref{exact sequence j}) is also splitting.

      Now we show that   (\ref{exact sequence h}) is splitting.    Recall that we have the following commutative diagram (by Corollary \ref{Sigma is canonical model}, $\sigma_3$ is surjective):
        $$\xymatrix    {0 \ar[r] & \mathcal{L}_2\otimes V_1 \ar[r]^{\beta_2\otimes id} \ar[d]^j \ar[dr]^h  & S^2(V_1)\otimes V_1 \ar[r]^{\sigma_2\otimes id} \ar[d]^{\gamma} &V_2\otimes V_1 \ar[r] \ar[d]^{\mu_{2,1}}  &0 \\
         0 \ar[r] & \mathcal{L}_3 \ar[r]^{\beta_3} & S^3(V_1) \ar[r]^{\sigma_3} &V_3 \ar[r]    & 0
         }$$

       First we  show that the surjection  $\gamma$ is splitting.

              Note that $$S^2(V_1)\otimes V_1$$
              $$=S^2(E_{[0]}(3,1))\otimes E_{[0]}(3,1)\oplus E_{[0]}(3,1)\otimes N\otimes E_{[0]}(3,1) \oplus S^2(E_{[0]}(3,1))\otimes N $$
              $$ \oplus N^{\otimes 2}\otimes E_{[0]}(3,1)
                \oplus E_{[0]}(3,1)\otimes N^{\otimes 2} \oplus N^{\otimes 3}$$
             $$S^3(V_1)=S^3(E_{[0]}(3,1))\oplus    S^2(E_{[0]}(3,1))\otimes N     \oplus E_{[0]}(3,1)\otimes N^{\otimes 2}   \oplus N^{\otimes 3}$$

            By the definition of $\gamma$, we can decompose the surjection $\gamma$ as $\gamma_1\oplus \gamma_2\oplus \gamma_3\oplus \gamma_4$, where
             $$\gamma_1:   S^2(E_{[0]}(3,1))\otimes E_{[0]}(3,1) \twoheadrightarrow   S^3(E_{[0]}(3,1)),$$
             $$\gamma_2:   E_{[0]}(3,1)\otimes N\otimes E_{[0]}(3,1) \oplus S^2(E_{[0]}(3,1))\otimes N  \twoheadrightarrow      S^2(E_{[0]}(3,1))\otimes N,$$

              $$\gamma_3:   N^{\otimes 2}\otimes E_{[0]}(3,1)  \oplus E_{[0]}(3,1)\otimes N^{\otimes 2}  \twoheadrightarrow  E_{[0]}(3,1)\otimes N^{\otimes 2},$$

              $$\gamma_4:       N^{\otimes 3}   \twoheadrightarrow N^{\otimes 3} .$$

               By \cite{CP06} p. 1035, $\gamma_1$ is splitting. Since moreover $\gamma_4$ is an isomorphism, it suffices to show that $\gamma_2$ and  $\gamma_3$ are splitting. By \cite{CP06} p. 1035, we have
                 $$S^2(E_{[0]}(3,1))=E_{[0]}(3,2) \oplus  E_{[0]}(3,2) ,$$
               $$  E_{[0]}(3,1)\otimes E_{[0]}(3,1)   =S^2(E_{[0]}(3,1))\oplus E_{[0]}(3,2) .$$

               In particular,  each  side of $\gamma_2$ (resp. $\gamma_3$)  is a direct sum of $E_{[0]}(3,2)\otimes N$'s  (resp. $E_{[0]}(3,1)\otimes N^{\otimes 2}$'s).
      By   Lemma  \ref{nonzero map indecomposable vector bundles  slope less than} and Corollary \ref{isomorphism of V W},   $\gamma_2(E_{[0]}(3,2)\otimes N)$ (resp. $\gamma_3(E_{[0]}(3,1)\otimes N^{\otimes 2})$) is either $0$ or one direct summand of  right side of $\gamma_2$ (resp.  $\gamma_3$). Hence $\gamma_2$ and $\gamma_3$ are   splitting.

             By Remark \ref{T2 is 0 for genus bigger than 4} (iii), the exact sequence (\ref{short exact sequence sigma 2}) is splitting. In particular, we have $h=\gamma\circ(\beta_2\otimes id)=\gamma_3\circ (\beta_2\otimes id)$. Since  $(\beta_2\otimes id)(\mathcal{L}_2\otimes V_1)\cong \mathcal{L}_2\otimes V_1=E_{[0]}(3,1)\otimes N^{\otimes 2}$ and $h$ is injective,  $h(\mathcal{L}_2\otimes V_1)=\gamma_3((\beta_2\otimes id)(\mathcal{L}_2\otimes V_1))$  must be a direct summand $E_{[0]}(3,1)\otimes N^{\otimes 2}$ of $S^3(V_1)$.

                  Therefore the exact sequence (\ref{exact sequence h}) is splitting, and consequently (\ref{exact sequence j})  is also  splitting.
                   \end{proof}
  
 Set $\mathcal{L}_3':=\cokernel(j), \tilde{V}_3:=\cokernel(h)$. Then   $\mathcal{L}_3'$ is a line bundle of degree $1$ and
 \begin{equation} \label{tilde of V3}
 \tilde{V}_3=(\bigoplus_{i=1,2}\mathcal{O}_B(0)) \oplus (\bigoplus_{j=1}^8\mathcal{M}_j(0))\oplus (\bigoplus_{l=1,2}E_{[0]}(3,2)\otimes N)
  \end{equation}
  where $\mathcal{M}_j$  are the line bundles with $\mathcal{M}_j^{\otimes 3}\cong \mathcal{O}_B, \mathcal{M}_j\ncong \mathcal{O}_B$.

 \begin{remark}\label{another proof for T3 is 0}
 (i)  Here we can give another proof for Corollary \ref{Sigma is canonical model} without using Lemma \ref{Takahashi's lemma generated in degree 1}:

   Since $\beta_3'$   is an injection, by Lemma \ref{injection  less than maxslope} we have $\slope(\mathcal{L}_3')=1+\deg\mathcal{T}_3\leq \maxslope(\tilde{V}_3)=1$. Hence $\deg \mathcal{T}_3=0$, i.e. $\mathcal{T}_3=0$.    By a result of Cai (cf. \cite{Cai95}) $\mathcal{R}(f)$ is generated in degree $\leq 3$. Therefore $\mathcal{R}(f)$ is generated in degree 1.

   (ii) From  (i) we see $\deg \mathcal{L}_3'=1=\maxslope(\tilde{V}_3)$. By (\ref{tilde of V3}) and Lemma \ref{nonzero map indecomposable vector bundles  slope less than}, we see that $\mathcal{L}_3'\cong \mathcal{O}_B(0)$ or $\mathcal{L}_3'\cong \mathcal{M}_i(0)$ for some $1\leq i\leq 8$.
\end{remark}

Now we have the following short  exact sequence
           \begin{equation}  \label{short exact sequence sigma 3' }
               0\rightarrow \mathcal{L}'_3 \stackrel{\beta_3'} {\rightarrow}   \tilde{V}_3 \stackrel{\sigma_3}{\rightarrow} V_3\rightarrow 0.
         \end{equation}

 The injection $\beta_3'$ defines a divisor $Y$ in $Q$, which is nothing but $\Sigma$ (see Remark \ref{Y equas Sigma}).

  \begin{pro}   \label{Sigma is complete intersection}
     $\Sigma$ is a complete intersection of type $(2,3)$ in $\mathbb{P}(V_1)$.
  \end{pro}

    \begin{proof}
     Note that $\Sigma\subset Q$ is a divisor defined by a section in
     $$H^0(Q,\mathcal{O}_Q(3)\otimes \pi_Q^*(\mathcal{L}_3')^{-1})\cong H^0(B, \tilde{V}_3\otimes (\mathcal{L}_3')^{-1}).$$
      Since the exact sequence (\ref{exact sequence h}) is splitting, we see that the restriction map
         $$H^0(B,S^3(V_1)\otimes (\mathcal{L}_3')^{-1})\rightarrow H^0(B, \tilde{V}_3\otimes (\mathcal{L}_3')^{-1})$$
          is surjective. Since moreover
     $$H^0(\mathcal{O}_{\mathbb{P}(V_1)}(3)\otimes \pi^*(\mathcal{L}_3')^{-1})\cong   H^0(B,S^3(V_1)\otimes (\mathcal{L}_3')^{-1}),$$
       we see
       $$|\mathcal{O}_Q(3)\otimes \pi_Q^*(\mathcal{L}_3')^{-1}|\cong |\mathcal{O}_{\mathbb{P}(V_1)}(3)\otimes \pi^*(\mathcal{L}_3')^{-1}|_Q.$$
       In particular   $\Sigma \in |\mathcal{O}_Q(3)\otimes \pi_Q^*(\mathcal{L}_3')^{-1}|$  is       the restriction of some divisor $X\in |\mathcal{O}_{\mathbb{P}(V_1)}(3)\otimes \pi^*(\mathcal{L}_3')^{-1}|$ to $Q$, i.e. $\Sigma=X\cap Q$.
       \end{proof}

	\begin{lemma} \label{N is of order 3}
      $N$ is a nontrivial torsion line bundle of order 3, i.e. $N^{\otimes 3}\cong \mathcal{O}_B, N\ncong \mathcal{O}_B$.
 \end{lemma}
 \begin{proof}
   Recall that we have the following commutative diagram
   $$\xymatrix{ S\ar[r]^w \ar[dr]^f & \Sigma\subset\mathbb{P}(V_1) \ar[d]^\pi \\
   ~ & B }$$
   where $w$ is the relative canonical map of $f$.
    Since $w$ is base point free and  $\Sigma$ has at most rational double points as singularities, we see  $w^*\mathcal{O}_\Sigma(T)\cong \omega_S\cong w^*\omega_{\Sigma}$.
   Since $\Sigma=\tilde{\Sigma}\cap Q$, we have
   $$\omega_\Sigma=\mathcal{O}_\Sigma(K_{\mathbb{P}(V_1)}+X+Q)=\mathcal{O}_\Sigma(T)\otimes \pi|_\Sigma^*(N^{-1}(0)\otimes
  (\mathcal{L}_3')^{-1}).$$

   (Note $\omega_{\mathbb{P}(V_1)}\cong \mathcal{O}_{\mathbb{P}(V_1)}(-4T)\otimes \pi^*N(0)$ and  $\mathcal{O}_{\mathbb{P}(V_1)}(Q+X)=\mathcal{O}_{\mathbb{P}(V_1)}(4T)\otimes \pi^*(N^{-2}\otimes (\mathcal{L}_3')^{-1})$.)

       Hence we get $\pi|_\Sigma^*(N^{-1}(0)\otimes(\mathcal{L}_3')^{-1})\cong \mathcal{O}_\Sigma$ and consequently  $N^{-1}(0)\otimes(\mathcal{L}_3')^{-1}\cong \mathcal{O}_B$, i.e. $N\cong (\mathcal{L}_3')^{-1}(0)$. By Remark \ref{another proof for T3 is 0} (ii) and the fact that $N\ncong \mathcal{O}_B$,  we see that $N\cong \mathcal{M}_i$ for some  $1\leq i\leq 8$.
    \end{proof}

Combining Proposition \ref{Sigma is complete intersection} and Lemma \ref{N is of order 3}, we get Theorem  \ref{Sigma is contained in the family}.

  \begin{cor}  \label{mathcal M is an irreducible component}
     Surfaces satisfying condition $(\star)$ give a 4-dimensional  irreducible subset $\mathcal{M}$ (cf. Definition \ref{defn of family M}) of $\mathcal{M}_{1,1}^{4,4}$. Moreover the closure $\overline{\mathcal{M}}$ of $\mathcal{M}$ is an  irreducible component of  $\mathcal{M}_{1,1}^{4,4}$.
\end{cor}
  \begin{proof}
    By Theorem \ref{Sigma is contained in the family}, surfaces satisfying condition $(\star)$ are in one to one correspondence with surfaces in $M$. Hence  their image in  $\mathcal{M}_{1,1}^{4,4}$ is exactly  $\mathcal{M}$.

     Note that  $p_g,q,K^2,g$ are all deformation invariants. By  \cite{CC91} Theorems 1.2 and 1.4,    ``$V_1$ is decomposable" is a topological invariant condition, hence it is also a deformation invariant condition. Since moreover ``the general Albanese fibre is nonhypereliptic" is an open condition, we see that   condition ($\star$) is an open condition. Hence a small deformation of such a surface also satisfies condition $(\star)$. Therefore   $\overline{\mathcal{M}}$   is an  irreducible component of  $\mathcal{M}_{1,1}^{4,4}$.
  \end{proof}

  \begin{remark} \label{pg=q=2 also 1to1 corresponence}
  All the arguments in this section   work for surfaces in Remark \ref{examples with pg is 2} (i). i.e. surfaces $S$ satisfying the following condition

   $(\star \star)$    $S$ is a minimal surface with $p_g=q=2,K^2=4$ endowed with a genus 4   nonhyperelliptic  fibration $f: S\rightarrow B$  over an elliptic curve $B$ such that   $V_1=f_*\omega_S$ is decomposable.

   are in one to one correspondence with surfaces in the family in Remark \ref{examples with pg is 2}.

   In particular,  these surfaces give a 4-dimensional irreducible subset in $\mathcal{M}_{2,2}^4$, the Gieseker moduli space of minimal surfaces with $p_g=q=2,K^2=4$. (see \cite{CMP14} for more details about minimal surfaces with $p_g=q=2,K^2=4$).
 \end{remark}

	\section{Appendix}
In this appendix, we list some lemmas that we use in section 4. 
 Here we always assume that $V$ and $V'$  are two vector bundles over an elliptic curve $B$.

\begin{lemma}[Atiyah \cite{Ati57} Lemma 15 and Theorem 5]
\label{Atiyah lemma 15 and theorem 5}

If   $V$ is   indecomposable,   then $h^0(V)=d$ if $d>0$; $h^0(V)=0$ if $d<0$;  $h^0(V)=0$ or $1$ if $d=0$. Moreover   if $d=0$, then $h^0(V)=1$ if and only if $\det V\cong \mathcal{O}_B$.

\end{lemma}
		
	\begin{lemma}  \label{nonzero map indecomposable vector bundles  slope less than}
	If $V$ and $V'$ are both indecomposable  and   there is a nonzero map $\xi: V\rightarrow V'$, then
 $\slope(V)\leq \slope(V')$.
\end{lemma}
 \begin{proof}
            Since $V$ and $V'$ are indecomposable, if  $\slope(V)>\slope(V')$, then each direct summand of $V'\otimes V^\vee$ has slope $\slope(V')-\slope(V)<0$. In particular, each direct summand has   degree $<0$, hence it has no nonzero global section by  Lemma \ref{Atiyah lemma 15 and theorem 5}. Hence $Hom(V,V')\cong H^0(V'\otimes V^\vee)=0$, a contradiction. Therefore we have $\slope(V)\leq \slope(V')$.
\end{proof}

\begin{lemma} \label{indecomposable is semistable}
          If  $V$ is indecomposable, then it is semi-stable. If moreover $(\rank V, \deg V)=1$, then $V$ is stable.
  \end{lemma}
  \begin{proof}
    The first assertion follows easily from Lemma \ref{nonzero map indecomposable vector bundles  slope less than}. Now assume further $(\rank V, \deg V)=1$.  Assume that  $V_0$ is  subbundle of $V$ with $0<\rank V_0<\rank V$. Since $V$ is semi-stable, we have $\slope(V_0)=\frac{\deg V_0}{\rank V_0}\leq \slope(V)=\frac{\deg V}{\rank V}$. Since $(\rank V, \deg V)=1$, the inequality must be strict, i.e.   $\slope(V_0)<\slope(V)$. Therefore $V$ is stable.
 \end{proof}

  \begin{cor} \label{isomorphism of V W}
    Conditions as in Lemma \ref{nonzero map indecomposable vector bundles  slope less than}. If moreover $\rank V=\rank V'=r, \deg V=\deg V'=d$ and $(r,d)=1$, then $\xi$ is an isomorphism.
  \end{cor}

   \begin{proof}
    By Lemma \ref{indecomposable is semistable}, $V$ and $V'$ are both stable.   Since $\xi(V)$ is a nonzero subbundle of $V'$, we have $\slope(\xi(V)) \leq \slope(V')$ and equality holds if and only if $\xi(V)=V'$. On the other hand, since $\xi(V)$ is a quotient bundle of $V$, we have $\slope(\xi(V)) \geq \slope(V)$ an equality holds if and only if $\xi(V)\cong V$. Since moreover $\slope(V)=\slope(V')$,    we get $\slope(\xi(V))=\slope (V)=\slope(V')$. Hence   $V\cong \xi(V)=V'$, i.e. $\xi$ is an isomorphism. 
\end{proof}

  Let $V$ be a vector bundle over an elliptic  curve $B$, we  define
  $$\maxslope(V):=\max\{\slope(V^i)|V^i  \text{ is a direct summand of } V\}.$$
   Then we have
\begin{lemma} \label{injection  less than maxslope}
  Let $V$ and $V'$ be two   vector bundles over an elliptic curve $B$. If there is an injection $\xi: V\rightarrow V'$, then $\slope(V)\leq \maxslope(V)\leq \maxslope(V')$.
\end{lemma}

\begin{proof}

 It is obvious that $\slope(V)\leq \maxslope(V)$. Assume $V^1$ to be an indecomposable  direct summand of $V$ (possibly $V^1=V$) such that $\slope(V^1)=\maxslope(V)$. Then  the restriction map
 $\xi|_{V^1}: V^1\rightarrow V'$ is nonzero. Hence there is an indecomposable direct summand  $(V')^1$ of $V'$ such that   the composition map $pr\circ \xi|_{V^1}: V^1\rightarrow V'\rightarrow (V')^1$ is nonzero, where $pr: V'\rightarrow (V')^1$ is the natural projection.   By Lemma \ref{nonzero map indecomposable vector bundles  slope less than}, we have $\slope(V^1)\leq \slope((V')^1)$. Hence $\maxslope(V)=\slope(V^1)\leq \slope((V')^1)\leq \maxslope(V')$.
 \end{proof}

\vspace{3ex}
$\mathbf{Acknowledgements.}$
I am   indebted to  Professor Fabrizio Catanese    for introducing me   this topic and  also for  some helpful suggestions. I am grateful to   Professor Jinxing Cai for his encouragement and some good suggestions. Thanks also goes to Yi Gu for some useful  discussion.

	%\bibliography{/Users/sihong/Reference/journalname,/Users/sihong/Reference/graph}

\vspace{3ex}
School of Mathematics Sciences, Beijing University, Yiheyuan Road 5, Haidian District, Beijing 100871,  People's Republic of China

E-mail address: 1201110022@pku.edu.cn

\end{document}